\documentclass[a4paper,10pt]{amsart}

\usepackage{amssymb}
\usepackage{amsmath, amscd}
\usepackage{amsthm}
\usepackage[all]{xy}
\usepackage{graphicx}
\usepackage{enumerate}

\makeatletter
\newtheorem*{rep@theorem}{\rep@title}
\newcommand{\newreptheorem}[2]{%
\newenvironment{rep#1}[1]{%
 \def\rep@title{#2 \ref{##1}}%
 \begin{rep@theorem}}%
 {\end{rep@theorem}}}
\makeatother

\theoremstyle{plain}
\newtheorem{theorem}{Theorem}[section]
\newreptheorem{theorem}{Theorem}
\newtheorem{lemma}[theorem]{Lemma}
\newtheorem{proposition}[theorem]{Proposition}
\newreptheorem{proposition}{Proposition}

\newreptheorem{corollary}{Corollary}

\newtheorem*{btheorem}{Basis Theorem}
\newtheorem*{stheorem}{Stallings' Folding Theorem}

\theoremstyle{definition}
\newtheorem{definition}[theorem]{Definition}
\newtheorem{example}[theorem]{Example}

\theoremstyle{remark}
\newtheorem{remark}[theorem]{Remark}

\def\aut{{\rm{Aut}}}

\def\ssm{\smallsetminus}
\def\SL{{\rm{SL}}}
\def\GL{{\rm{GL}}}

\def\G{\mathcal{G}}
\def\r{\rho}

\def\ian{{\rm{IA}}_n}
\def\Fix{{\rm{Fix}}}

\address{Department of Mathematics, University of Utah, Salt Lake City, UT 84112}
\email{wade@math.utah.edu}

\author{Richard D. Wade}

\title{Folding free-group automorphisms}

\begin{document}

\begin{abstract}
We describe an algorithm that uses Stallings' folding technique to decompose an element of $\aut(F_n)$ as a product of Whitehead automorphisms (and hence as a product of Nielsen transformations.) This algorithm is known to experts, but has not yet appeared in the literature. We use the algorithm to give an alternative method of finding a finite generating set for the subgroup of $\aut(F_n)$ that fixes a subset $Y$ of the basis elements, and the subgroup that fixes each element of $Y$ up to conjugacy. We show that the intersection of this latter subgroup with $IA_n$ is also finitely generated.
\end{abstract}

\maketitle

\section{Introduction}

The idea of controlling cancellation between words in a group can be traced along a line of thought spanning the twentieth century, from Nielsen's 1921 paper \cite{N} showing that a finitely generated subgroup of a free group is free, through to the combinatorial and geometric methods in small cancellation theory now prevalent in the study of group actions on CAT(0) and hyperbolic complexes. In the free group, Nielsen's method of \emph{reduction} was extended and given a topological flavour by Whitehead, who looked at sphere systems in connected sums of copies of $S^1 \times S^2$ \cite{MR1503309}. Whitehead's idea of \emph{peak reduction} was refined and recast in combinatorial language by Rapaport~\cite{MR0131452}, Higgins and Lyndon \cite{MR0340420}, and McCool \cite{MR0396764}. There is a good description of this viewpoint in Lyndon and Schupp's book on combinatorial group theory \cite{MR0577064}.

Peak reduction is very powerful. Given a finite set $Y$ of elements in $F_n$, McCool~\cite{MR0396764} gives an algorithm to obtain finite presentations of $\Fix(Y)$ and $\Fix_c(Y)$, the subgroups of $\aut(F_n)$ that fix $Y$ pointwise, and fix each element of $Y$ up to conjugacy, respectively. Culler and Vogtmann's work on Outer Space shows that such subgroups also satisfy higher finiteness properties \cite{MR830040}.

The generating sets for $\Fix(Y)$ and $\Fix_c(Y)$ are built up out of \emph{Whitehead Automorphisms}. These are automorphisms of two types. The first consists of the group $W_n$ of automorphisms that permute and possibly invert elements of a fixed basis. So if $F_n$ is generated by $X=\{x_1,\ldots,x_n\}$, then for each $\phi \in W_n$ there exists $\sigma \in S_n$ and $\epsilon_1,\ldots,\epsilon_n \in \{-1,1\}$ such that $\phi(x_i)=x_{\sigma(i)}^{\epsilon_i}$. For the second type, we pick an element $a \in X \cup X^{-1}$ and for each basis element, we either pre-multiply by $a$, post-multiply by $a^{-1}$, or do both of these things. Traditionally this is defined by taking a subset $A \subset X \cup X^{-1}$ such that $a \in A$ and $a^{-1} \not \in A$, and defining $(A,a) \in \aut(F_n)$ by \begin{equation*}(A,a)(x_j)= \begin{cases} x_j &\text{if $x_j=a^{\pm1}$} \\ a^{\alpha_j}x_ja^{-\beta_j} &\text{if $x_j \neq a^{\pm 1}$} \end{cases},\end{equation*}
where $\alpha_j=\chi_A(x_j)$ and $\beta_j=\chi_A(x_j^{-1})$. 

Beyond the work of Nielsen and Whitehead, a third approach to reduction in free groups comes from Stallings \cite{MR695906}, who cast Nielsen reduction in terms of folds on graphs. Since `Topology of finite graphs' appeared in 1983, folding has become a key tool in geometric group theory, notably in its applications to graphs of groups and their deformation spaces \cite{GL,KWM,F}, and to the dynamics of free group automorphisms (and endomorphisms) \cite{BH,FH,DV}. In this paper we give an account of how folding gives an algorithm to decompose an automorphism as a product of Whitehead automorphisms. This algorithm is hinted at by Stallings \cite[Comment 8.2]{MR695906}, and will be familiar to many authors who have used his techniques, but no explicit account appears in the literature. However, Carette's thesis \cite{Carette} uses Stallings folds to give not only finite generation, but finite presentations for automorphism groups of free products of groups (under a natural hypothesis on the factors).

The chief advantage of folding over peak reduction is the ease of application: folding a graph is less complicated than searching through a list of possible Whitehead automorphisms (a list that grows exponentially with $n$). Moreover, folding gives an intuitive, pictorial way of looking at the decomposition.  The proofs in this paper are geared towards making it easy to produce such calculations by hand or with a computer.

Finite generation for many subgroups of the form $\Fix(Y)$ and $\Fix_c(Y)$ also follows very naturally from this description. In Section \ref{s:applications} we show that if $Y$ is a subset of our preferred basis for $F_n$ then the folding algorithm implies that $\Fix(Y)$ and $\Fix_c(Y)$ are generated by the Whitehead automorphisms that lie in $\Fix(Y)$ and $\Fix_c(Y)$, respectively (see Figure \ref{Graphs} for a quick idea of how this is done.) We apply this result to show that when $Y$ is a subset of a basis the intersection of $\Fix_c(Y)$ with $IA_n$, the subgroup of $\aut(F_n)$ acting trivially on $H_1(F_n)$, is also finitely generated. In particular, we give a description of Magnus' proof that $IA_n$ is finitely generated.

\section{Graphs, Folding, and associated automorphisms} \label{s:graphs}

The fundamental group of a graph gives a pleasant pictorial description of the free group, and can be thought of as both a topological and a combinatorial construction. In this paper we will focus on the latter approach, borrowing most of our notation from Serre's book \cite{serretrees}. Proofs in this first section will either be sketched or omitted. 

\subsection{The fundamental group of a graph}


\begin{definition}
A \emph{graph} $G$ consists of a tuple $(\text{E}G,\text{V}G,inv,\iota,\tau)$ where $\text{E}G$ and $\text{V}G$ are sets and $inv:\text{E}G\rightarrow \text{E}G$, $\iota,\tau : \text{E}G \rightarrow \text{V}G$  are maps which satisfy
\begin{align*}
 inv(e) &\neq e \\
inv(inv(e)) &= e \\
\iota(inv(e))&=\tau(e).
\end{align*}
$\text{E}G$ is said to be the \emph{edge set} of $G$ and $\text{V}G$ the \emph{vertex set} of $G$. For an edge $e \in \text{E}G$ we write $inv(e)=\bar{e}$, and say that $\iota{(e)},\tau(e)$ are the \emph{initial} and \emph{terminal} vertices of $e$ respectively.
\end{definition}


   A \emph{path} $p$ in $G$ is either a sequence of edges $e_1, \ldots e_k$ such that $\iota{(e_{i+1})}=\tau{(e_i)}$, or a single vertex $v$. Let $\text{P}G$ be the set all paths.  The functions $inv$, $\iota$ and $\tau$ extend to $\text{P}G$; in the case where $p$ is a sequence of edges we define $\iota(p)=\iota(e_1)$, $\tau(p)=\tau(e_k)$ and $\bar{p}=\bar{e}_k, \ldots \bar{e}_1$, and evaluate these functions at $v$ if $p$ is a single vertex $v$. We say that $G$ is connected if for any two vertices $v,w$ there exists a path $p$ such that $\iota{(p)}=v$ and $\tau{(p)}=w$. If $\tau(p_1)=\iota(p_2)$ we define $p_1.p_2$ to be the concatenation of the two sequences. We define an equivalence relation $\sim$ on P$G$ by saying two paths $p_1$, $p_2$ are equivalent if and only if one can be obtained from the other by insertion and deletion of a sequence of pairs of edges of the form $(e,\bar{e})$.  We say that a path $p$ is \emph{reduced} if there are no consecutive edges of the form $(e,\bar{e})$ in $p$. 
   
\begin{proposition}
Every element of $\text{\emph{P}}G / \sim$  is represented by a unique reduced path.  For $p \in \text{\emph{P}}G$ we let $[p]$ denote the reduced path in the equivalence class of $p$.
\end{proposition}
   
   The set of reduced paths that begin and end at a vertex $v$ in $G$ form a group that we shall denote $\pi_1(G,v)$, the \emph{fundamental group of $G$ based at $v$}.  Multiplication is defined as follows --- if $p,q$ are reduced paths, then $p \cdot q = [p.q]$. The identity element is the path consisting of the single vertex $v$, and the inverse of a reduced path $p$ is the path $\bar{p}$. A path $p_{vw}$ connecting vertices $v$ and $w$ in $G$ induces an isomorphism $[p] \mapsto [p_{vw}.p.\overline{p_{vw}}]$ between $\pi_1(G,w)$ and $\pi_1(G,v)$. A \emph{subgraph} of $G$ is given by subsets of $\text{E}G$ and $\text{V}G$ which are invariant under the operations $inv$ and $\iota$.  A connected graph $T$ is called a \emph{tree} if $\pi_1(T,v)$ is trivial for a (equivalently, any) vertex $v$ of $T$.  We say that $T$ is a \emph{maximal tree} in a connected graph $G$ if $T$ is a subgraph of $G$, $T$ is a tree, and the vertex set of $T$ is $\text{V}G$.  Such a tree always exists. Given a base point $b$ in a 
connected graph $G$ and a maximal tree $T$, there exists a unique reduced path $p_v$ from $b$ to $v$. An \emph{orientation} of a subgraph $G' \subset G$ is a set $\mathcal{O}$ that contains exactly one element of $\{e,\bar{e}\}$ for each element of $G'$. An \emph{ordered orientation} of $G'$ is an orientation $\mathcal{O}$ of $G'$ with an enumeration of the set $\mathcal{O}$.
   
\begin{proposition}\label{p:tree orientation}
Let $T$ be a maximal tree in a connected graph $G$ with chosen base point $b$. Then we can define an orientation $\mathcal{O}(T,b)$ of $T$ by saying that $e \in \mathcal{O}(T,b)$ if an only if $e$ occurs as an edge in a path $p_v$ for some $v$. \end{proposition}

Geometrically, this is the orientation one obtains by drawing arrows on edges `pointing away from $b$.' The main use of maximal trees and orientations will be to give a basis for $\pi_1(G,b)$. The following theorem is key to the rest of the paper, so we will give it a name:

\begin{btheorem} \label{p:basis}
Let $T$ be a maximal tree in a connected graph $G$  with chosen base point $b$. Let $\{e_1, \ldots, e_n\}$ be an ordered orientation of $G \ssm T$.  Let $$l_i=p_{\iota(e_i)}e_i\overline{p_{\tau(e_i)}}.$$ $\pi_1(G,b)$ is freely generated by $l_1,\ldots,l_n$. Given any loop $l$ based at $b$, we may write $[l]$ as a product of the generators as follows: remove the edges of $l$ contained in $T$ to obtain a sequence $e_{i_1}^{\epsilon_1},\ldots, e_{i_k}^{\epsilon_k}$, where $i_j \in \{1,\ldots,n\}$ and $\epsilon_j \in \{1,-1\}$. Then $$[l]=[l_{i_1}^{\epsilon_1}\cdots l_{i_k}^{\epsilon_k}].$$
\end{btheorem} 
  
Thus, once we have a maximal tree and an ordered orientation of the edges outside of this tree, the Basis Theorem gives us a method for constructing an ordered free generating set of $\pi_1(G,b)$. It also tells us how to write any element of $\pi_1(G,b)$ as a product of these generators. We may determine when a subgraph of $G$ is a maximal tree as follows:

\begin{lemma} \label{tree lemma}
Let $G$ be a connected graph, T a subgraph of $G$ and $b$ a vertex of $G$.  Then T is a maximal tree if and only if:

\begin{enumerate}
\item T contains $2(|\text{\emph{V}}G| - 1)$ edges.
\item For each vertex $v$ of $G$ there exists a reduced path $p_v$ from $b$ to $v$ in $T$.
\end{enumerate}
\end{lemma}

\subsection{Folding maps of graphs}

From now on we shall assume that all graphs are connected. A \emph{map of graphs} $f:G \rightarrow \Delta$ is a map that takes edges to edges, vertices to vertices and satisfies $f(\bar{e})=\overline{f(e)}$ and $f(\iota(e))=\iota(f(e))$ for every edge in $G$.  For a vertex $v$ of $G$ the map $f$ induces a group homomorphism $f_*:\pi_1(G,v) \rightarrow \pi_1(\Delta,f(v))$.  If $f_*$ is an isomorphism for some (equivalently, any) choice of vertex of $G$, we say that $f$ is a \emph{homotopy equivalence}. If $f$ is bijective on E$G$ and V$G$ then we say $f$ is a \emph{graph isomorphism}. The \emph{star} of a vertex $v$ is defined to be \[ St(v, G)=\{e \in \text{E}G : \iota(e)=v \}. \]

If $f$ is a map of graphs then for each vertex $v$ in $G$ we obtain a map $f_v:St(v,G) \rightarrow St(f(v),\Delta)$ by restricting $f$ to the edges in $St(v,G)$.  We say that $f$ is an \emph{immersion} if $f_v$ is injective for each vertex of $G$, and we say that $f$ is a \emph{covering} if $f_v$ is bijective for each vertex of $G$.  If for some vertex $v$ the map $f_v$ is not injective, Stallings \cite{MR695906} introduced a method called \emph{folding} for improving the map $f$: take edges $e_1$ and $e_2$ in $St(v,G)$ such that $f_v(e_1)=f_v(e_2)$ and form a quotient graph $G '$ by identifying the pairs $\{e_1, e_2\}$,  $\{\bar{e}_1,\bar{e}_2\}$ and $\{\tau(e_1),\tau(e_2)\}$ in $G$ to form quotient edges $e'$, $\bar{e}'$ and a quotient vertex $v'$.

There are then induced maps $q:G \rightarrow G'$ and $f':G' \rightarrow \Delta$ such that $f' \cdot q = f$.  We call this process a \emph{folding} of $G$. If $v$ is a vertex in $G$ the map $q_*:\pi_1(G,v) \rightarrow \pi_1(G',q(v))$ is surjective and $f_*(\pi_1(G,v))=f_*'(\pi_1(G',q(v)))$.

\begin{stheorem}[\cite{MR695906}] 
Let $f:G \rightarrow \Delta$ be a map of graphs, and suppose that $G$ is finite and connected. 
\begin{enumerate}
\item If $f$ is an immersion then $f_*$ is injective.
\item If $f$ is not an immersion, there exists a finite sequence of foldings $G=G_0 \rightarrow G_1 \rightarrow G_2 \ldots \rightarrow G_n$ and an immersion $G_n \rightarrow \Delta$ such that the composition of the above maps is equal to $f$.
\end{enumerate}
\end{stheorem}

\begin{proof}[Sketch proof]
If $f$ is an immersion, then reduced paths are sent to reduced paths of the same length. Hence $f_*$ is injective. For the second part, we iterate the folding described above to obtain a sequence of graphs with the required properties. This process must eventually end as $G$ is finite, and folding reduces the number of edges in a graph.
\end{proof}

There are four different types of fold that can occur, which we illustrate in Figure~\ref{fig:FoldingDiagram}. If $f_*$ is injective only folds of type 1 or 2 occur.  In case 3 the loop $e_1,\bar{e}_2$ is non-trivial in the original graph, but mapped to the trivial element in the quotient, and in 4 the loops $e_1$ and $e_2$ are distinct but mapped to homotopic loops in the quotient. 

\begin{figure} [ht]
\centering
\includegraphics{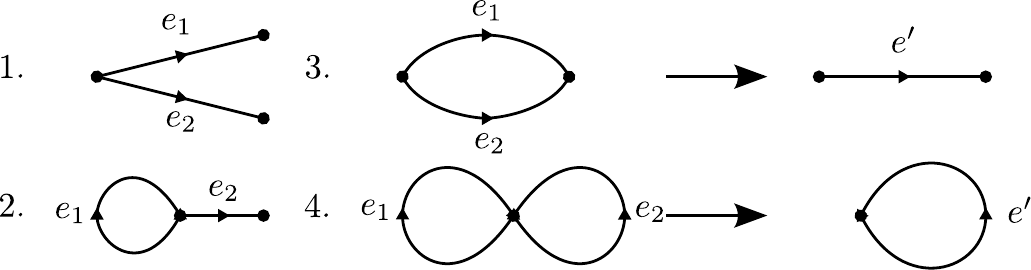}
\caption{Possible folds of a graph}
\label{fig:FoldingDiagram}
\end{figure}
\subsection{Branded graphs and their associated automorphisms}

We may identify $F_n$ with the fundamental group of a fixed graph, R$_n$:

\begin{definition}
The \emph{rose with n petals}, R$_n$ is defined be the graph with edge set $\text{ER}_n=\{x_1,\ldots,x_n\}\cup\{\bar{x}_1,\ldots,\bar{x}_n\}$, a single vertex $b_R$ with $\iota(e)=\tau(e)=b_R$ for each edge $e$ in $\text{ER}_n$ and $inv$ taking $x_i \rightarrow \bar{x}_i$.  We identify $F_n$ with $\pi_1(\text{R}_n,b_R)$ by the map taking each generator $x_i$ of $F_n$ to the path consisting of the single edge with the same name.
\end{definition}

Suppose that $f:G \to R_n$ is a homotopy equivalence. Let $T$ be a maximal tree of $G$, let $b$ be a vertex of $G$, and let $\{e_1,\ldots,e_n\}$ be an orientation and an ordering of the elements of $G \ssm T$. We call the tuple $\mathcal{G}=(G,f,b,\{e_1,\ldots, e_n\})$ a \emph{branded graph}. If we are given $G$ and $f$, then we say that a choice of a base point $b$ and an ordered orientation of a complement of a maximal tree in $G$ is a \emph{branding}. As $b$ and $\{e_1,\ldots,e_n\}$ determine a choice of basis of $\pi_1(G,b)$, every branded graph has an \emph{associated automorphism} of $F_n$ defined by: $$\phi_\mathcal{G}(x_i)=f_*(l_i),$$

where $l_i$ is the loop $p_{\iota(e_i)}.e_i.\overline{p_{\tau(e_i)}}$ described in Proposition \ref{p:basis}. Topologically, the choice of basepoint $b$ and edges $\{e_1,\ldots,e_n\}$ determines a homotopy equivalence $(R_n,b_{R}) \xrightarrow{h_{\mathcal{G}}} (G,b)$ given by mapping $x_i$ over $l_i$. Then $\phi_\mathcal{G}$ is the automorphism $f_*{h_{\mathcal{G}}}_*$:
\[ \xymatrix{
 &\pi_1(G,b) \ar[d]^{f_*}\\
\pi_1(R_n,b_{R}) \ar[ru]^{{h_\mathcal{G}}_*} \ar[r]^{\phi_\mathcal{G}} & \pi_1(R_n,b_{R}) }\]
\begin{example}\label{e:auto}
If $\phi \in \aut(F_n)$ and $\phi(x_i)=w_i$ for all $i$, let $G$ be the graph that is topologically a rose, with the $i$th loop subdivided into $|w_i|$ edges. Let $f:G\to R_n$ be the homotopy equivalence given by mapping the $i$th loop to the path given by $w_i$ in $R_n$. Let $b$ be the vertex in the centre of the rose, and for each $i$ choose an edge $e_i$ in the $i$th loop oriented in the direction of the word $w_i$. If $\mathcal{G}=(G,f,b,\{e_1,\ldots,e_n\})$ then $l_i$ i the $i$th loop, hence $\phi_\mathcal{G}=\phi$.
\end{example}

Of particular importance is the situation when $f$ is an immersion:

\begin{lemma} \label{immersion lemma}
Let $f:G\to R_n$ be a homotopy equivalence and an immersion. Then $f$ is an isomorphism, and for any branding $\mathcal{G}$ associated to $G$, $f$, we have $\phi_\mathcal{G} \in W_n$.
\end{lemma}

\begin{proof}
If $f$ is an isomorphism of graphs, then $\phi_\mathcal{G} \in W_n$ for any branding -- each $e_i$ forms a loop in $G$, so there exists $\sigma \in S_n$ such that each $e_i$ is sent to $x_{\sigma(i)}^{\epsilon_i}$ for some $\epsilon_i \in \{-1,1\}$ that depends on $i$. It remains to show that if $f$ is an immersion and a homotopy equivalence then $f$ is an isomorphism. One way to see this is as follows: if $f$ is an immersion, there exists a graph $G'$ containing $G$ and a map $f':G' \to R_n$ which covers $R_n$ (e.g. \cite{MR695906}, Theorem 6.1). However, $f_*$ is surjective, so this cover is degree 1, and $G'=G\cong R_n$.
\end{proof}

\section{The algorithm}\label{s:algorithm}

The algorithm for writing an arbitrary element of $\phi \in \aut(F_n)$ as a product of Whitehead automorphisms proceeds as follows. One first picks a branded graph $\G$ such that $\phi_{\mathcal{G}}=\phi$; to be definite we take the one described in Example~\ref{e:auto}. If $f$ is not an immersion then a fold occurs and, since $f$ is a homotopy equivalence, it can only be one of the two types shown in Figure \ref{FoldingDiagram2}.

\begin{figure}[ht]
 \centering
\includegraphics{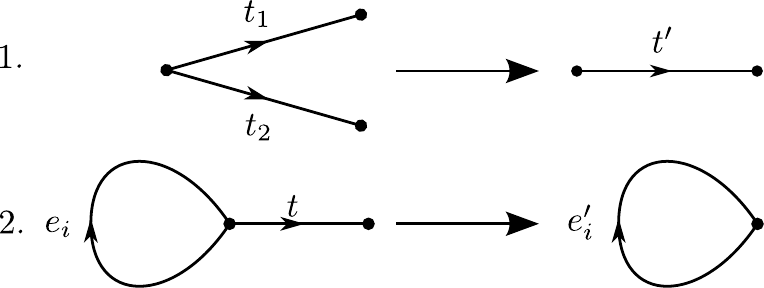}
\caption{Possible folds when $f$ is a homotopy equivalence}
\label{FoldingDiagram2}
\end{figure}

If, as the labelling in Figure~\ref{FoldingDiagram2} suggests, the folding edges with two distinct endpoints are in the maximal tree T, then we obtain a branding $\mathcal{G}'$ of the folded graph. In the first case, the associated automorphisms $\phi_\G$ and $\phi_{\G'}$ are identical (Proposition \ref{type1fold}), and in the second, they differ by a Whitehead automorphism of the form $(A,a)$ that may be read off from the structure of $T$ (Proposition \ref{type2fold}).

It may happen that one of $t_1,t_2$, or $t$ does not lie in $T$. In this case we can swap this edge with an edge already lying in $T$ (see Section \ref{ss:tree}), to obtain a new tree $T'$ and a new branding $\mathcal{G}'$. Again, $\phi_\mathcal{G}$ and $\phi_{\mathcal{G}'}$ differ by a Whitehead automorphism of the form $(A,a)$ that may be read off from the swap (Proposition \ref{tree sub}). After at most two such swaps, we can ensure that the folding edges with two distinct endpoints lie in $T$, and proceed as above (Remark \ref{foldingremark}).

By Stallings' folding theorem, we obtain a finite sequence $\G =\G_1,\G_2, \ldots, \G_k$ of branded graphs $\G_j=(G_j,f_j, b_j, \{e_1^j,\ldots,e_n^j\})$ such that each $f_j$ is a homotopy equivalence, and $f_k$ is an immersion. By Lemma \ref{immersion lemma} we know that $f_k$ is an isomorphism and $\phi_{\G_k} \in W_n$. Then:
$$\phi=\phi_{\mathcal{G}_1}=\phi_{\G_k}(\phi_{\G_{k}}^{-1}\phi_{\G_{k-1}})\cdots(\phi_{\G_3}^{-1}\phi_{\G_2})(\phi_{\G_2}^{-1}\phi_{\G_1})$$
is a decomposition of $\phi$ as a product of Whitehead automorphisms. Throughout this paper we shall assume that $\aut(F_n)$ acts on $F_n$ on the left, so that in the above decomposition we apply $\phi_{\G_2}^{-1}\phi_{\G_1}$ first, then $\phi_{\G_3}^{-1}\phi_{\G_2}$, etc.

If we count one step as a (possibly trivial) tree substitution, followed by a fold, then each step reduces the number of combinatorial edges of the graph by two (an $e$ and an $\bar{e}$). If the initial graph has $2m$ edges, then as $R_n$ has $2n$ edges we will obtain a decomposition of $\phi$ after $m-n$ steps. If $\phi(x_i)=w_i$ and we start with the graph given in Example~\ref{e:auto}, then our algorithm will terminate after $(\sum_{i=1}^n|w_i|)-n$ steps.

It remains to give a detailed description of the process of folding and exchanging edges in maximal trees.

\subsection{Folding edges contained in T} \label{ss:fold}

Suppose $q:G \to G'$ is a fold from Figure  \ref{FoldingDiagram2}. The map $f$ factors through $q$, inducing a homotopy equivalence $f':G' \to R_n$ such that $f=f' \cdot q$. Let $b',e_1',\ldots,e_n'$ be the images of $b,e_1,\ldots,e_n$ respectively under $q$. Then $\mathcal{G}'=(G',f',b',\{e_1',\ldots,e_n'\})$ is a branding of $G'$. The only thing to check is that $T'=G'\ssm\{e_1',\overline{e'_1},e_2',\overline{e_2'},\ldots,e_n,\overline{e_n'}\}$ is a maximal tree of $G'$. The subgraph $T'$ contains $2(|VG'|-1)$ edges as a fold of type 1 or 2 reduces the number of vertices in a graph by one, and the number of edges in a graph by two. Let $v'$ be a vertex of $G'$. Take a vertex $v$ of $G$ such that $q(v)=v'$. In the case of a type 1 fold, the path $[q(p_v)]$ is a reduced path from $b'$ to $v'$ lying in $T'$, and in the case of a type 2 fold, if we remove all occurrences of $e_i'$ from $q(p_v)$, then reduce, we obtain a path from $b'$ to $v'$ lying in $T'$. Hence by Lemma \ref{tree lemma}, we know that $T'$ is 
a maximal tree of $G'$.
Let $p_v$ be the unique reduced path from $b$ to $v$ in $T$ and let $l_1,\ldots,l_n$ be the generators of $\pi_1(G,b)$ given by $b$ and $\{e_1,\ldots,e_n\}$. Let $l_1',\ldots,l_n'$ be the generators of $\pi_1(G',b')$ given by $b'$ and $\{e_1', \ldots,e_n'\}$. As $f_*=f'_*q_*$, we may find the difference between the automorphisms $\phi_\mathcal{G}$ and $\phi_\mathcal{G'}$ by finding a decomposition of $q_*(l_i)$ in terms of the $l_i'$. 

\begin{proposition} \label{type1fold}
Suppose that $q$ is a fold of type 1, where the folded edges $t_1$ and $t_2$ lie in $T$. Then $\phi_\mathcal{G}=\phi_{\mathcal{G}'}$.
\end{proposition}

\begin{proof} For each path $l_i$, the only edge $q(l_i)$ crosses that does not lie in $T'$ is $e_i'$. By the Basis Theorem, we have $q_*(l_i)=l_i'$. Hence \begin{equation*}\phi_{\mathcal{G}'}(x_i)=f_*'(l_i')=f_*'(q_*(l_i))=f_*(l_i)=\phi_{\mathcal{G}}(x_i). \qedhere \end{equation*}\end{proof}

\begin{proposition} \label{type2fold}
Let $q$ be a fold of type 2, where we identify an edge $t$ in $T$ with the edge $e_i$ (and identify $\bar{t}$ with $\bar{e}_i$). Let $\mathcal{O}(T,b)$ be the orientation of $T$ given by Proposition \ref{p:tree orientation}. Let $$\epsilon=\begin{cases} 1 &\text{if $t \in \mathcal{O}(T,b)$} \\ -1 &\text{if $\bar{t} \in \mathcal{O}(T,b)$.}\end{cases} $$ Define $A \subset X \cup X^{-1}$ such that $x_i^\epsilon \in A$, $x_i^{-\epsilon} \not \in A$ and \begin{align*} x_j \in A &\Leftrightarrow \text{$p_{\iota(e_i)}$ crosses $t$ or $\bar{t}$} \\ x_j^{-1} \in A &\Leftrightarrow \text{$p_{\tau(e_i)}$ crosses $t$ or $\bar{t}$.}\end{align*} Then $\phi_\mathcal{G}=\phi_{\mathcal{G}'}\cdot (A,x_i^{\epsilon})$.
\end{proposition}

\begin{proof} We prove this result for $t \in \mathcal{O}(T,b)$, the other case being similar. If $t \in \mathcal{O}(T,b)$, then $t$ may appear at most once in a path $p_v$, however $\bar{t}$ may not. Note that: \begin{align*} q(l_j)&=q(p_{\iota(e_j)}e_j\overline{p_{\tau(e_j)}}) \\ &=q(p_{\iota(e_j)}).e_j'.q(\overline{p_{\tau(e_j)}}). \end{align*} Removing all the edges of $q(l_j)$ not in $T'$ leaves a sequence of the form $(e_j')$, $(e_i',e_j')$, $(e_i',e_j',\overline{e_i'})$ or $(e_j', \overline{e_i'})$, where $e_i'$ proceeds $e_j'$ if and only if $t$ lies in $p_{\iota(e_j)}$, and $\overline{e_i'}$ follows $e_j'$ if and only if $t$ lies in $p_{\tau(e_j)}$. As $e_i$ is a loop, $p_{\iota(e_i)}=p_{\tau(e_i)}$, and therefore this sequence is either $(e_i')$ or $(e_i',e_i',\overline{e_i'})$. Therefore $q_*(l_i)=l_i'$ and it follows that $\phi_\mathcal{G}(x_i)=\phi_{\mathcal{G}'}(x_i)$. If $j\neq i$ then by the Basis Theorem we have $[q(l_j)]=[l_i']^{\alpha_j}.[l_j'].[l_i']^{-\beta_j}$ where $\alpha_j=\chi_A(x_j)
$ and $\beta_j=\chi_A(x_j^{-1})$. Hence \begin{align*} \phi_{\mathcal{G}'} \cdot (A,x_i)(x_j)&=\phi_{\mathcal{G}'}(x_i^{\alpha_j}x_jx_i^{-\beta_j}) \\ &=f'_*([l_i']^{\alpha_j}.[l_j'].[l_i']^{-\beta_j}) \\ &=f'_*q_*(l_j) \\ &=f_*(l_j) \\ &=\phi_{\mathcal{G}}(x_j) \qedhere \end{align*}\end{proof}

\subsection{Swapping edges into a tree}\label{ss:tree}

Suppose that we would like to fold in a branded graph as in Figure 2, but an edge $t_1$, $t_2$ or $t$ lies outside the maximal tree. Then either this edge or its inverse is equal to $e_i$ for some $i$. The edge $e_i$ has distinct endpoints, so $p_{\iota(e_i)} \neq p_{\tau(e_i)}$. Let $a$ be the shared initial segment of these paths. Either $p_{\iota(e_i)} \ssm a$ or $p_{\tau(e_i)} \ssm a$ is non-empty. Choose an edge $e_i'$ such that either $e_i' \in p_{\iota(e_i)} \ssm a$ or $\overline{e_i'} \in p_{\tau(e_i)} \ssm a$. By a similar approach to the one used in Section \ref{ss:fold} one can check that $T'=G\ssm\{e_1,\overline{e_1},e_2,\overline{e_2},\ldots,e_i',\overline{e_i'},\ldots, e_n,\overline{e_n}\}$ is a maximal tree of $G$, so that $\mathcal{G'}=(G,f,b,\{e_1,\ldots,e_i',\ldots,e_n\})$ is a branding of $G$.

\begin{figure}[ht]
 \centering
\includegraphics{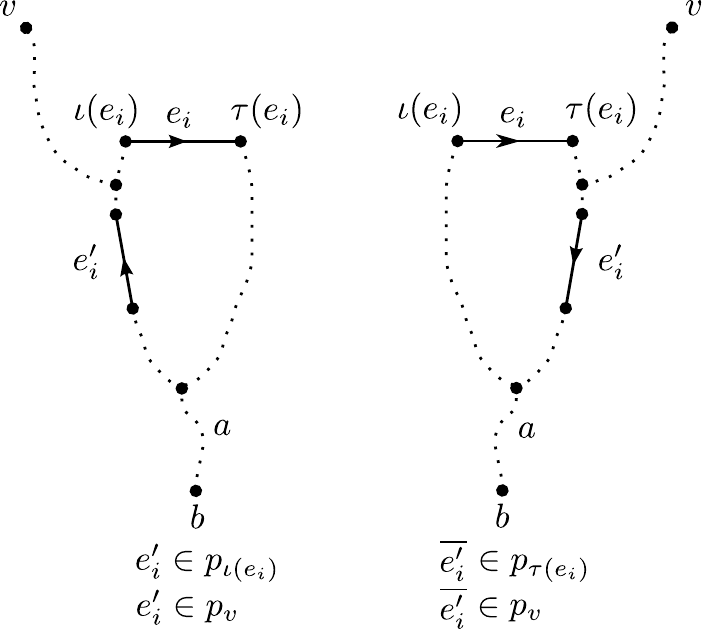}
\caption{Changing maximal trees.}
\label{TreeSubstitution}
\end{figure}

\begin{proposition} \label{tree sub}
Let $\mathcal{G}'$ be the branding obtained by swapping an edge as described above and depicted in Figure 3. Define $$\epsilon=\begin{cases} 1 & \text{if $e_i' \in p_{\iota(e_i)}$} \\ -1 &\text{if $e_i' \in \overline{p_{\tau(e_i)}}.$} \end{cases} $$ Now define $A \subset X \cup X^{-1}$ to be such that $x_i^{\epsilon} \in A$, $x_i^{-\epsilon} \not \in A$ and \begin{align*} x_j \in A &\Leftrightarrow \text{$p_{\iota(e_j)}$ crosses $e_i'$ or $\overline{e_i'}$} \\ x_j^{-1} \in A &\Leftrightarrow \text{$p_{\tau(e_j)}$ crosses $e_i'$ or $\overline{e_i'}$.} \end{align*} Then $\phi_\mathcal{G}=\phi_{\mathcal{G}'} \cdot (A,x_i^{\epsilon}).$
\end{proposition}

\begin{proof}
The proof is analogous to the proof of Proposition \ref{type2fold}. Let $l_1',\ldots,l_n'$ be the new basis of $\pi_1(G,b)$ given by $b$ and $\{e_1,\ldots,e_i',\ldots,e_n\}$. By reading off the edges that lie outside of $T'$ crossed by the paths $l_j$ we find that $l_i=l_i'$ and for $j\neq i$ we have $l_j=[l_i'^{\epsilon{\alpha_j}}.l_j'.l_i'^{-\epsilon \beta_j}]$, where $\alpha_j=\chi_A(x_j)$ and $\beta_j=\chi_A(x_j)$. It follows that  $\phi_\mathcal{G}=\phi_{\mathcal{G}'} \cdot (A,x_i^{\epsilon}).$ \end{proof}

\begin{remark} \label{foldingremark}
If we are looking at a fold of the first type in Figure \ref{FoldingDiagram2}, we would like both edges $t_1$ and $t_2$ to lie in the maximal tree T. If we move one edge $t_1$ into the maximal tree through the method described above, the edge $t_2$ may still lie outside the maximal tree. We would like to add it in without removing $t_1$.  We are only unable to do this if $t_1$ and $\bar{t}_1$ are the only elements of $p_{\iota(t_2)} \ssm a$ and $p_{\tau(t_2)}\ssm a$. This means that $\{p_{\iota(t_2)},p_{\tau(t_2)}\}$ is either the set $\{a,a.t_1\}$ or the set $\{a,a.\bar{t}_1\}$.  These cases would contradict either $\iota(t_1)=\iota(t_2)$ or $\tau(t_1)\neq\tau(t_2).$ 
\end{remark}

\section{Applications}\label{s:applications}

In this section we show how the algorithm described in Section \ref{s:algorithm} may be applied to find generating sets of subgroups of $\aut(F_n)$.

\subsection{Fixing generators.}

Let $\r_{ij},K_{ij},$ and $S_i$ be the elements of $\aut(F_n)$ defined by: \begin{align*} \r_{ij}(x_k)&=\begin{cases} x_ix_k &\text{if $k=i$} \\ x_k &\text{if $k\neq i$} \end{cases}, \\ K_{ij}(x_k)&=\begin{cases} x_ix_kx_i^{-1} &\text{if $k=i$} \\ x_k &\text{if $k\neq i$} \end{cases}, \\ S_i(x_k)&=\begin{cases} x_i^{-1} &\text{if $k=i$} \\ x_k &\text{if $k \neq i$} \end{cases}. \end{align*}

These elements are called a \emph{right Nielsen automorphism}, a \emph{partial conjugation} and an \emph{inversion} respectively. Any Whitehead automorphism can be written as a product of the above elements. Let  $\text{Fix}(\{x_{m+1},\ldots,x_n\})$ be the subgroup of $\aut(F_n)$ consisting of elements that fix $x_{m+1},\ldots,x_n$ pointwise, and let  $\text{Fix}_c(\{x_{m+1},\ldots,x_n\})$ be the subgroup of $\aut(F_n)$ that takes each element of the set $\{x_{m+1},\ldots,x_n\}$ to a conjugate of itself. 

\begin{figure}[ht]
\centering
\includegraphics{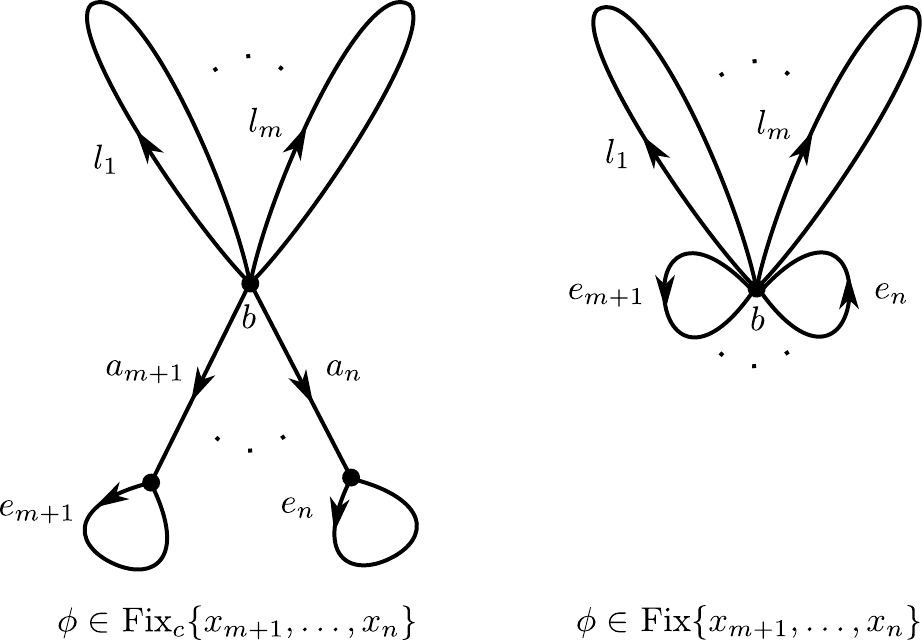}
\caption{The Construction of $G$ in Theorem \ref{generatingsets1}}
\label{Graphs}
\end{figure}

\begin{theorem} \label{generatingsets1}
Let $Y=\{x_{m+1},\ldots,x_n\}$ be a subset of our preferred basis for $F_n$. The subgroups $\Fix(Y)$ and $\Fix_c(Y)$ are generated by the Whitehead automorphisms that lie in $\Fix(Y)$ and $\Fix_c(Y)$ respectively. In terms of Nielsen automorphisms, generating sets for $\Fix(Y)$, $\Fix_c(Y)$ are given by 
\begin{align*}                                                                                                                                                                                                                             \mathcal{A}_m&=\{S_i,\r_{ij}:1 \leq i \leq m, 1 \leq j \leq n\},\\
\mathcal{B}_m&= \mathcal{A}_m\cup\{K_{ij}:m+1\leq i\leq n, 1 \leq j \leq n\},                                                                                                                                                                                                                                                                                         \end{align*}
respectively.
\end{theorem}

\begin{proof} 
Let $\phi \in \text{Fix}_c(Y)$ and let $G$ be a graph constructed as follows:  take a single vertex $b$ and a loop $l_j$ consisting of $|\phi(x_j)|$ edges about $b$ for $x_1,\ldots,x_m$. We have $\phi(x_j)=w_jx_jw_j^{-1}$ for $x_{m+1},\ldots,x_n$ --- add a path $a_j $ containing $|w_j|$ edges to $b$ for each $j$, and attach an edge loop $e_j$ to the end of each of these paths.  We can then define $f:G \rightarrow \text{R}_n$ by mapping each loop $l_j$ to the edge path $\phi(x_j)$, each path $a_j$ to the edge path $w_j$, and the edge loops $e_{m+1},\ldots,e_n$ to the edges $x_{m+1},\ldots,x_n$ respectively (see Figure \ref{Graphs}).  Pick an edge $e_j$ in each $l_j$ oriented in the direction of the word $\phi(x_j)$ being spelt out by $l_j$.  Then $\phi$ is the automorphism associated to the branded graph $\mathcal{G}=(G,f,b,\{e_1,\ldots,e_n\})$.  We apply the algorithm described in Section \ref{s:algorithm} to write $\phi$ as a product of Whitehead automorphisms. Let $\G =\G_1,\G_2, \ldots, \G_k$ be the 
sequence of branded graphs $\G_j=(G_j,f_j, b_j, \{e_1^j,\ldots,e_n^j\})$ obtained. Let $e_i$ be an edge in $\{e_{m+1},\ldots,e_n\}.$ Then each $e_i^j$ is a loop, and will never be swapped into a maximal tree, so $e_i^j \rightarrow e_i^{j+1}$ at each step in the folding process. As $\iota(e_i^j)=\tau(e_i^j)$, we have $p_{\iota(e_i^j)}=p_{\tau(e_i^j)}$ at each step, so by Propositions \ref{type2fold} and \ref{tree sub} the only Whitehead automorphisms of the form $(A,a)$ that occur in the decomposition of $\phi_\mathcal{G}$  take $x_j$ to a conjugate.  Also, $\phi_{\mathcal{G}_k} \in W_n$ fixes $x_{m+1},\ldots,x_n$.  Hence the Whitehead automorphisms that lie in $\Fix_c(Y)$ generate $\text{Fix}_c(Y)$. In the case where $x_{m+1},\ldots,x_n$ are completely fixed by $\phi$, the loops $e_{m+1}^j,\ldots,e_n^j$ are at the basepoint of each graph in the folding process, therefore Propositions \ref{tree sub} and \ref{type2fold} tell us every Whitehead automorphism that occurs in the decomposition of $\phi$ will fix $x_
{m+1},\ldots,x_n$. To obtain the generating sets in terms of Nielsen automorphisms one checks that each Whitehead automorphism that lies in $\Fix(Y)$ may be written as a product of elements of $\mathcal{A}_m$, and that each Whitehead automorphism that lies in $\Fix_c(Y)$ may be written as a product of elements that lie in $\mathcal{B}_m$.
\end{proof}

\subsection{Fix$_c(\{x_{m+1},\ldots,x_n\}) \cap \ian$} \label{s:ian}
Let $IA_n$ be the subgroup of $\aut(F_n)$ that acts trivially on the abelianisation of $F_n$. Magnus \cite{MR1555401} showed that $IA_n$ is generated by elements of the form:
\begin{align*} K_{ij}(x_l)&=\begin{cases} x_jx_ix_j^{-1} & i=l \\
               x_l & i \neq l
              \end{cases} \\
 K_{ijk}(x_l)&=\begin{cases} x_i[x_j,x_k] & i=l \\
               x_l & i \neq l,
              \end{cases} \end{align*}
where $i$, $j$, and $k$ are distinct. Again we take $Y=\{x_{m+1},\ldots,x_n\}$ to be a subset of our fixed basis for $F_n$. We shall use an adaptation of Magnus' proof to show that $\Fix_c(Y) \cap \ian$ is generated by Magnus' generators that lie in $\Fix_c(Y)$.  (This includes Magnus' theorem in the case $Y=\emptyset$.)

We use the following general observation: let $G$ be a group, $H$ a normal subgroup of $G$ and $\overline{G} = G/H$.  Let $A$ be a generating set of $G$, let $\overline{A}$ be the image of $A$ in $\overline{G}$, and let $R$ be a set of words in $G$ such that $\overline{G}$ has the presentation $\overline{G}= \langle \overline{A} | \overline{R} \rangle.$  Then $H$ is the subgroup of $G$ normally generated by the elements of $R$. If $B$ is a subset of $H$ such that $B$ generates a normal subgroup of $G$ and this subgroup contains $R$, then $B$ is a generating set of $H$.   

We shall proceed as follows: we first find a presentation for the group 
\begin{equation*} G_m = \left\{ \begin{pmatrix} A & 0 \\ B & I \end{pmatrix} : A \in \GL_m(\mathbb{Z}), B \in M_{n-m, m}(\mathbb{Z}) \right\} \leq \GL_n(\mathbb{Z}) \end{equation*}
in Proposition \ref{linearpresentations}.  The group $G_m$ is the image of $\Fix_c(Y)$ under the map $\Theta: \aut(F_n)\rightarrow \GL_n(\mathbb{Z})$.  Hence the kernel of this restricted map is $\Fix_c(Y) \cap \ian$. It only remains to check that all the relations lie in the subgroup generated by our chosen set, and that this set generates a normal subgroup of $\Fix_c(Y)$.

Let $M_{ij}$ be the matrix taking the value 1 in the $(i,j)$th entry, and zeroes everywhere else.  When $i \neq j$ let $E_{ij}=I+M_{ij}$, and let $T_i=I -2M_{ii}$, the matrix that takes the value $-1$ in the $(i,i)$th entry, $1$ in the other diagonal entries, and zero everywhere else. The group $G_m$ is isomorphic to the semidirect product $\mathbb{Z}^{(n-m)m} \rtimes \GL_m(\mathbb{Z})$, where \begin{align*}\mathbb{Z}^{(n-m)m} &\cong \left\{ \begin{pmatrix} I & 0 \\ B & I \end{pmatrix}  \in G_m \right\} \\ \GL_m(\mathbb{Z}) &\cong \left\{ \begin{pmatrix} A & 0\\ 0 & I \end{pmatrix} \in G_m  \right\},\end{align*}

therefore to find a presentation of $G_m$ it is sufficient to find presentations for $\mathbb{Z}^{(n-m)m}$ and $\GL_m(\mathbb{Z}),$ and relations that describe the action of $\GL_m(\mathbb{Z})$ on $\mathbb{Z}^{(n-m)m}$ by conjugation. The $\mathbb{Z}^{(n-m)m}$ part of $G_m$ has the obvious presentation $\langle\, E_{ij}\, | \, R_{1,m} \, \rangle, $ where $m+1 \leq i \leq n$, $1 \leq j \leq m$ and $R_{1,m}$ contains the commutators of these elements. The $\GL_m(\mathbb{Z})$ part of $G_m$ has a presentation $\langle\, T_1, E_{ij}\, | \, R_{2,m} \, \rangle,$ where $1 \leq i,j \leq m$ and 
\[ R_{2,m}=\left\{ \begin{aligned}&T_1^2 &\\
&(E_{12}E_{21}^{-1}E_{12})^4  &\\
&E_{12}E_{21}^{-1}E_{12}E_{21}E_{12}^{-1}E_{21} &\\
&[E_{ij},E_{kl}] &i\neq k, j \neq l \\ 
&[E_{ij},E_{jk}]E_{ik}^{-1}  &i,j,k \text{ distinct} \\
&[T_1,E_{ij}] &i\neq 1, j \neq 1\\
&T_1E_{ij}T_1E_{ij} &1 \in \{i,j\}
\end{aligned} \right\} . \]

This is easily deduced from the Steinberg presentation of $\SL_n(\mathbb{Z}),$ which can be found in  \cite[pages 81--82]{MR0349811}, and the decomposition $\GL_n(\mathbb{Z})=\SL_n(\mathbb{Z})\rtimes \langle T_1 \rangle.$  There is an exception for $m=1$, which has the much simpler presentation $\langle\, T_1\, |\, T_1^2\, \rangle.$  The relations that occur from the action of $\GL_m(\mathbb{Z})$ on $\mathbb{Z}^{(n-m)m}$ by conjugation are of the form:
\[ R_{3,m}=\left\{ \begin{aligned} 
&E_{ij}E_{kl}E_{ij}^{-1}=E_{kl} &i \neq k \\
&E_{ij}E_{kl}E_{ij}^{-1}=E_{kj}^{-1}E_{kl}  &\text{$i=l$ and $i,j,k$ are distinct}\\
&T_1E_{kl}T_1=E_{kl} &k,l \neq 1\\
&T_1E_{kl}T_1=E_{kl}^{-1} &1 \in \{k,l\}
\end{aligned} \right\} \]

where $E_{ij}$ is taken over elements in our copy of $\GL_m(\mathbb{Z})$ and $E_{kl}$ is taken over elements in our copy of $\mathbb{Z}^{(n-m)m}$.  Summarising:

\begin{proposition} \label{linearpresentations} 
\mbox{}
\[\langle T_1, E_{ij} \quad 1 \leq i \leq n,\; 1 \leq j \leq m\; | R_{1,m} \cup R_{2,m} \cup R_{3,m} \rangle\]
is a presentation of $G_m$.
\end{proposition}

\begin{theorem}
$\ian \cap \Fix_c(\{x_{m+1},\ldots,x_n\})$ is generated by the set $$\mathcal{C}_m=\{K_{ij} : 1 \leq i \leq n,\; 1\leq j \leq n \} \cup \{K_{ijk} : 1 \leq i \leq m,\; 1 \leq j,k \leq n\}.$$
\end{theorem}

\begin{proof}
We can remove the elements $S_2,\ldots,S_m$ from the generating set $\mathcal{B}_m$ of the group $\Fix_c(Y),$ as $S_i=S_1\r_{1i}\r_{i1}^{-1}S_1\r_{1i}^{-1}S_1\r_{1i}S_1\r_{i1}\r_{1i}^{-1}S_1$, to make a smaller generating set $\mathcal{B}_m'$. Then $\mathcal{B}_m'$ maps onto the generating set of $G_m$ given in Proposition \ref{linearpresentations} by taking $\r_{ij} \rightarrow E_{ji}$, $S_1 \rightarrow T_1$. The elements $K_{ij}$ are taken to the identity matrix. From the discussion given above, it suffices to show that $\langle \mathcal{C}_m \rangle$ is a normal subgroup of $\Fix_c(Y)$ that contains the lift of each element of $R_{1,m} \cup R_{2,m} \cup R_{3,m}$ obtained by swapping $E_{ij}$ with $\r_{ji}$ and $T_1$ with $S_1$.  It is not hard to check that the lift of each relation to $\aut(F_n)$ lies in $\langle \mathcal{C}_m \rangle$.  To prove normality it is sufficient to show that the conjugate of every element of $\mathcal{C}_m$ by each element of $\mathcal{B}_m'\cup\mathcal{B}_m'^{-1}$ lies in $\
langle \mathcal{C}_m \rangle$.  Most of these computations are simple, except in the case of $\r_{pk}K_{kpq}\r_{pk}^{-1}$ and $\r_{pk}^{-1}K_{kpq}\r_{pk}$, which we write as products of elements of $\mathcal{C}_m$ below:
\begin{align*}
\r_{pk}K_{kpq}\r_{pk}^{-1}&=K_{qk}K_{qp}K_{pq}K_{pqk}K_{kp}K_{kpq}K_{kq}^{-1}K_{kp}^{-1}K_{qp}^{-1}K_{qk}^{-1} \\
\r_{pk}^{-1}K_{kpq}\r_{pk}&=K_{qk}^{-1}K_{qp}K_{pq}^{-1}K_{qp}^{-1}K_{kpq}K_{pqk}K_{qk}K_{kq}. \qedhere \end{align*} \end{proof}

\bibliographystyle{plain}
\bibliography{bib}{}

\end{document}